\documentclass{article}
\usepackage{amssymb,amsmath,bm,hyperref,geometry,graphics,graphicx,url,color}
\def\no{\noindent}
\def\pmatrix{\left(\begin{array}}
\def\endpmatrix{\end{array}\right)}
\def\no{\noindent}
\def\pmatrix{\left(\begin{array}}
\def\endpmatrix{\end{array}\right)}
\def\RR{\mathbb{R}}

\def\I{{\cal I}}

\def\P{{\cal P}}

\def\dd{\mathrm{d}}

\def\diag{\mathrm{diag}}

\newtheorem{theo}{Theorem}

\newtheorem{rem}{Remark}
\newtheorem{defi}{Definition}

\def\QED{\mbox{~$\Box{~}$}}

\def\bfzero{{\bm{0}}}

\def\bfgamma{{\bm{\gamma}}}

\def\bfphi{{\bm{\phi}}}

\def\bfGamma{{\bm{\Gamma}}}

\def\bfTheta{{\bm{\Theta}}}

\def\aa{\alpha}

\begin{document}

\title{A shooting-Newton procedure for solving fractional terminal value problems}

\author{Luigi Brugnano\,\footnote{
             Dipartimento di Matematica e Informatica ``U.\,Dini'',
             Universit\`a di Firenze,  Italy,
              \url{luigi.brugnano@unifi.it}, 
              \url{https://orcid.org/0000-0002-6290-4107}}
   \and  Gianmarco Gurioli\,\footnote{
                Dipartimento di Matematica e Informatica ``U.\,Dini'',
             Universit\`a di Firenze,  Italy,
              \url{gianmarco.gurioli@unifi.it},
              \url{https://orcid.org/0000-0003-0922-8119}}
   \and Felice Iavernaro\,\footnote{
           Dipartimento di Matematica,
           Universit\`a di Bari Aldo Moro, Italy,
            \url{felice.iavernaro@uniba.it},
            \url{https://orcid.org/0000-0002-9716-7370}}
}

\date{}
\maketitle

\begin{abstract}
In this paper we consider the numerical solution of  {\em fractional terminal value problems}: namely, {\em  terminal value problems for fractional differential equations}. In particular, the proposed method uses a Newton-type iteration which is particularly efficient when coupled with a recently-introduced step-by-step procedure for solving {\em fractional initial value problems}, i.e., {\em  initial value problems for fractional differential equations}. As a result, the method is  able to produce spectrally accurate solutions of fractional terminal value problems.  Some numerical tests are reported to make evidence of its effectiveness.

\medskip
\no{\bf Keywords:} fractional differential equations, fractional integrals, terminal value problems, Jacobi polynomials, Fractional Hamiltonian Boundary Value Methods, FHBVMs

\medskip
\no{\bf MSC:} 34A08, 65R20
\end{abstract}

\section{Introduction} 

Fractional differential equations have gained more and more importance in many applications: we refer, e.g., to the classical references \cite{Di2010,Po1999} for an introduction.

The present contribution is addressed for solving  {\em fractional terminal value problems} namely, {\em terminal value problems for fractional differential equations} (in short, {\em FDE-TVPs}) in the form
\begin{equation}\label{tvp}
y^{(\aa)}(t) = f(y(t)), \qquad t\in[0,T],\qquad y(T) = \eta\in\RR^m,
\end{equation}
where, for the sake of brevity, we have omitted the argument $t$ for $f$. 
Here, for ~$\aa\in(0,1)$, ~$y^{(\aa)}(t) \equiv D^\aa y(t)$~ is the Caputo fractional derivative:
\begin{equation}\label{Dalfa}
D^\aa g(t) = \frac{1}{\Gamma(1-\aa)} \int_0^t (t-x)^{-\aa} \left[\frac{\dd}{\dd x}g(x)\right]\dd x.
\end{equation}
The Riemann-Liouville integral associated to (\ref{Dalfa}) is given by:
\begin{equation}\label{Ialfa}
I^\aa g(t) = \frac{1}{\Gamma(\aa)}\int_0^t (t-x)^{\aa-1} g(x)\dd x.
\end{equation}

Usually, one solves {\em fractional initial value problems}, that is, {\em initial value problems for fractional differential equations}  (in short, {\em FDE-IVPs}) (see, e.g. \cite{DiFoFr2005,Ga2015,Ga2018,LYC16,Lu1985,SOG17}):
\begin{equation}\label{ivp}
y^{(\aa)}(t) = f(y(t)), \qquad t\in[0,T], \qquad y(0) = \rho_0\in\RR^m,
\end{equation}
whose solution is, under suitable assumptions on $f$,  
\begin{equation}\label{sol0}
y(t) = \rho_0 + I^{\aa} f(y(t)) \equiv \rho_0 +  \frac{1}{\Gamma(\aa)}\int_0^t (t-x)^{\aa-1} f(y(x))\dd x, \qquad t\in[0,T].
\end{equation}
However, under suitable hypothesis on $f$ and $T$, also the FDE-TVP is well-posed (see, e.g., \cite{FoMo2011} for the scalar case, and \cite{SWB2021}). Consequently, its numerical solution has been considered by many authors  (see, e.g., \cite{DiUl2023,FoMo2011,G2021,GK2021,LLZ2021,SWB2021}). In particular, the scalar case  of (\ref{tvp}) ($m=1$) allows using a shooting procedure coupled with the bisection method \cite{Di2015,FoMoRe2014} or, more recently, with the secant method \cite{DiUl2023}.  

However, as is clear by their very definition, both the above procedures {\em cannot} be applied for solving vector problems. Motivated by this drawback, in this paper, we propose an alternative approach, based on a straight Newton procedure, able to handle vector problems as well.

The procedure takes advantage of a recently introduced method for solving FDE-IVPs, able to obtain {\em spectrally accurate} approximations \cite{BBBI2024,BGI2024}. This latter approach has been derived as an extension of  Hamiltonian Boundary Value Methods (HBVMs), special low-rank Runge-Kutta methods originally devised for Hamiltonian problems (see, e.g., \cite{BI2016,BI2018}), and later extended along several directions (see, e.g., \cite{ABI2019,ABI2022-1,ABI2023,BBBI2024,BFCIV2022,BI2022}), including the numerical solution of FDEs. A main feature of HBVMs is the fact that they can gain spectrally accuracy, when approximating ODE-IVPs \cite{ABI2020,BIMR2019,BMR2019}, and such a feature has been recently extended to the FDE case \cite{BBBI2024,BGI2024}.

With this premise, the structure of the paper is as follows: in Section~\ref{newton} we sketch the shooting-Newton procedure for solving (\ref{tvp}), along with a corresponding simplified variant; in Section~\ref{qpol} we recall the main facts about the (possibly spectrally accurate) numerical solution of FDE-IVPs  recently proposed in \cite{BBBI2024}, with the extension for the shooting-Newton procedure; in Section~\ref{num} we report a few numerical tests,  including the case of vector problems; at last, a few conclusions are given in Section~\ref{fine}.

\section{The shooting-Newton procedure}\label{newton}

To begin with, let us introduce a perturbation result concerning the solution of the FDE-IVP (\ref{ivp}). In particular, let us denote by $y(t,\rho_0)$ the solution of this problem, in order to emphasize its dependence from the initial condition. The following result holds true.

\begin{theo}\label{fundmat} For $t\in[0,T]$, one has:
\begin{equation}\label{dy0}
\frac{\partial}{\partial \rho_0} y(t,\rho_0) = \Phi(t,\rho_0),
\end{equation} 
which is the solution of the fractional variational problem\,\footnote{As is usual, $f'(y)$ denotes the Jacobian matrix of $f(y)$.}
\begin{equation}\label{varpro}
\Phi^{(\aa)}(t,\rho_0) = f'(y(t,\rho_0))\Phi(t,\rho_0), \qquad t\in[0,T], \qquad \Phi(0,\rho_0) = I,
\end{equation}
explicitly given by:
\begin{equation}\label{Phi}
\Phi(t,\rho_0) = I + \frac{1}{\Gamma(\aa)}\int_0^t (t-x)^{\aa-1} f'(y(x,\rho_0))\Phi(x,\rho_0)\dd x.
\end{equation}
\end{theo}
\begin{proof}
In fact, from (\ref{Dalfa}) and (\ref{ivp}), one has: 
$$
 D^{\aa} \frac{\partial}{\partial \rho_0}y(t,\rho_0) \,=\, \frac{\partial}{\partial \rho_0} D^{\aa} y(t) \,=\, \frac{\partial}{\partial \rho_0} f(y(t,\rho_0)) \,=\, f'(y(t,\rho_0))\frac{\partial}{\partial \rho_0}y(t,\rho_0),
$$
and
$$\left.\frac{\partial}{\partial \rho_0}y(t,\rho_0)\right|_{t=0} \,=\, \frac{\partial}{\partial \rho_0}\rho_0 \,=\, I.$$
Consequently, (\ref{dy0})-(\ref{varpro}) follows and, therefore, also (\ref{Phi}) follows from (\ref{Ialfa}) and (\ref{sol0}).~\QED
\end{proof}\bigskip

\begin{rem}\label{wellposed}
Hereafter, in order to guarantee the well-posedness of problem (\ref{tvp}), if $y(0)=\rho^*$ is the initial value of (\ref{ivp}) fulfilling the FDE-TVP, i.e.,
\begin{equation}\label{soleta} y(T,\rho^*) = \eta,\end{equation}
we shall assume that (see (\ref{dy0})) 
\begin{equation}\label{PhiT0}
\det\left(\Phi(T,\rho^*)\right)\ne0.
\end{equation}
Assuming that $f$ is continuously differentiable in a neighborhood of the solution, in turn (\ref{PhiT0}) implies that
\begin{equation}\label{delta}
\exists\, \delta>0 ~s.t.~ \|\rho^*-\rho\|\le\delta ~\Rightarrow~ \det\left(\Phi(T,\rho)\right)\ne0.
\end{equation}
\end{rem}
\bigskip

The previous results allow us stating the shooting-Newton procedure for solving (\ref{tvp}) sketched in Table~\ref{alg1}, where a suitable stopping criterion has to be adopted. Moreover, the starting approximation $\rho_0$ for the shooting-Newton iteration has to be chosen in some way, possibly exploiting any additional information; conversely, the choice  $\rho_0=\eta$ (i.e., the final value in (\ref{tvp})) can be considered, as proposed in \cite{DiUl2023}.

\begin{table}[t]
\caption{Algorithm~\ref{alg1} -- the shooting-Newton procedure.}
\label{alg1}
\hrulefill
\begin{eqnarray*}
{\rm fix~} \rho_0\\
{\rm for ~~} \ell &=& 0,1,\dots\\
                    &&{\rm solve:}~\, y^{(\aa)}(t,\rho_\ell)=f(y(t,\rho_\ell)),\qquad\qquad~\,\, t\in[0,T],\qquad y(0,\rho_\ell) = \rho_\ell\\
                    &&~\,{\rm and} \quad     \Phi^{(\aa)}(t,\rho_\ell) = f'(y(t,\rho_\ell))\Phi(t,\rho_\ell),\quad t\in[0,T],\qquad \Phi(0,\rho_\ell)=I\\[2mm]
                    &&{\rm set:~} \quad \rho_{\ell+1} = \rho_\ell-\Phi(T,\rho_\ell)^{-1}\left[ y(T,\rho_\ell)-\eta\right]\\
                    {\rm end ~~}
\end{eqnarray*}
\hrulefill
\end{table}

\begin{rem} Though the procedure described in Algorithm~\ref{alg1} appears to be easily derived, at the best of our knowledge, it has not yet been considered for solving FDE-TVPs, so far. Moreover, the use of the variational problem, involved in its implementation and described in the next section, is  novel as well.
\end{rem}

The following straightforward convergence result holds true.

\begin{theo}\label{perron}  Assume that,  in a neighborhood of the solution $\xi=\rho^*$:   
\begin{itemize}
\item[\rm (i)] $f$ is continuously differentiable,

\smallskip
\item[\rm (ii)] $\Phi(T,\xi)^{-1}$ is differentiable.
\end{itemize}
Then, the shooting-Newton procedure given in Table~\ref{alg1} converges in a suitable neighborhood of $\rho^*$.
\end{theo}
\begin{proof}
In fact, from (\ref{soleta}) it follows that  $\rho^*$ is a fixed-point of the corresponding iteration function,
$$\Psi(\xi) := \xi - \Phi(T,\xi)^{-1}\left[y(T,\xi)-\eta\right],$$
whose Jacobian (recall (\ref{dy0})) vanishes at $\xi=\rho^*$. Consequently, from the Perron Theorem \cite[Corollary\,4.7.2]{LT1988}, exponential convergence is granted, in a suitable neighborhood of $\rho^*$.\,\QED
\end{proof}

\bigskip
Further, if convergent, the procedure converges quadratically. However, to prove this statement, we need to recall some well-known results about the Taylor theorem. 
In more detail, with reference to (\ref{dy0}), assume that $\Phi(T,\xi)$ is continuously differentiable in a suitable neighborhood of the solution. Then, by setting $y_i$ the $i$-th entry of $y$, for a given $\rho$ suitably close to $\rho_\ell$ there exists $\theta_i\in[0,1]$ such taht:
\begin{eqnarray*}
y_i(T,\rho) &=& y_i(T,\rho_\ell) + \left.\frac{\partial}{\partial\xi} y_i(T,\xi)\right|_{\xi=\rho_\ell} (\rho-\rho_\ell) \\
&&+ \frac{1}2 (\rho-\rho_\ell)^\top \left.\frac{\partial^2}{\partial\xi^2} y_i(T,\xi)\right|_{\xi=\rho_\ell +\theta_i(\rho-\rho_\ell)} (\rho-\rho_\ell), \qquad i=1,\dots,m,
\end{eqnarray*}
with $\frac{\partial^2}{\partial\xi^2} y_i(T,\xi)$ the Hessian matrix of $y_i(T,\xi)$. The previous relations can be written in vector form as follows:
$$y(T,\rho) = y(T,\rho_\ell) + \Phi(T,\rho_\ell)(\rho-\rho_\ell) + \frac{1}2\Phi'(T,\Sigma_\ell(\rho))\left((\rho-\rho_\ell),(\rho-\rho_\ell)\right),$$
with $$\Sigma_\ell(\rho) = \pmatrix{ccc} \rho_\ell+\theta_1(\rho-\rho_\ell),& \dots, &\rho_\ell+\theta_m(\rho-\rho_\ell)\endpmatrix\in\RR^{m\times m},$$
and $\Phi'(T,\Sigma_\ell(\rho))$ denoting the derivative of $\Phi$, whose $i$-th ``slice'' is evaluated in the $i$-th column of $\Sigma_\ell(\rho)$. With this premise, we can now state the following result.

\begin{theo}\label{quadratic}  Assume that,  in a neighborhood of the solution $\xi=\rho^*$,   
$\Phi(T,\xi)$ is continuously differentiable. Then, if convergent, the shooting-Newton procedure given in Table~\ref{alg1} converges quadratically.
\end{theo} 
\begin{proof} By using the notation about the Taylor theorem exposed before, one has:
 \begin{eqnarray*}
 0&=&y(T,\rho^*)-\eta \\
 &=& y(T,\rho_\ell)-\eta + \Phi(T,\rho_\ell)\left(\rho^*-\rho_\ell\right) + \frac{1}2\Phi'(T,\Sigma_\ell(\rho^*))\left( (\rho^*-\rho_\ell),(\rho^*-\rho_\ell)\right).
  \end{eqnarray*}
Consequently, considering that\,\footnote{We recall that (\ref{delta}) holds true.}
  $$\rho_{\ell+1} = \rho_\ell - \Phi(T,\rho_\ell)^{-1}[ y(T,\rho_\ell)-\eta],$$
and  setting $e_\ell=\rho^*-\rho_\ell$ the error at step $\ell$, one derives:
$$
 e_{\ell+1} = -\frac{1}2 \Phi(T,\rho_\ell)^{-1} \Phi'(T,\Sigma_\ell(\rho^*))\left( e_\ell,e_\ell\right).
$$
Passing to norms, one eventually obtains
$$\frac{\|e_{\ell+1}\|}{\|e_\ell\|^2} \le \frac{1}2 \|\Phi(T,\rho_\ell)^{-1}\|\,\| \Phi'(T,\Sigma_\ell(\rho^*))\|.$$
Consequently, 
$$\lim_{\ell\rightarrow\infty}\frac{\|e_{\ell+1}\|}{\|e_\ell\|^2}\le  \frac{1}2 \|\Phi(T,\rho^*)^{-1}\|\,\| \Phi'(T,\Sigma^*)\|,$$
 where $\Sigma^*$ now denotes the matrix with all the columns equal to $\rho^*$.\,\QED 
 \end{proof}

\bigskip
An interesting additional feature is given by the following result.

\begin{theo}\label{linear}
For problems in the form
\begin{equation}\label{linp}
y^{(\aa)} = A(t)y+b(t), \qquad t\in[0,T], \qquad y(T)=\eta\in\RR^m,
\end{equation}
with $A(t)$ and $b(t)$ continuous functions, 
the algorithm described in Table~\ref{alg1} converges in exactly one iteration.
\end{theo}
\begin{proof}
In fact, in such a case, the variational problem (\ref{varpro}) simplifies to
$$\Phi^{(\aa)}(t) = A(t) \Phi(t), \qquad t\in[0,T],\qquad \Phi(0) = I,$$
i.e., $\Phi(t)$ does not depend on the initial condition. Further, by using the same notation as above,
$$\left[y(t,\rho_0)-y(t,\rho^*)\right]^{(\aa)} = A(t)\left[ y(t,\rho_0)-y(t,\rho^*)\right], \qquad t\in[0,T],$$
whose solution is given by 
$$\left[ y(t,\rho_0)-y(t,\rho^*)\right] = \Phi(t)\left[\rho_0-\rho^*\right], \qquad t\in[0,T].$$
Consequently, at $t=T$ one has:
$$\left[ y(T,\rho_0) - \eta \right] = \Phi(T)\left[\rho_0-\rho^*\right].$$
That is,\footnote{We recall that the assumption $\det(\Phi(T))\ne0$ must clearly hold.}
$$\rho^* = \rho_0 -\Phi(T)^{-1}\left[ y(T,\rho_0) - \eta \right],$$
so that convergence is gained in exactly one iteration, since the r.h.s. amounts to the very first iteration of Algorithm~\ref{alg1} used for solving (\ref{linp}).\,\QED
\end{proof}

\subsection{A simplified Newton-iteration}

As is well-known, sometimes it can be computationally convenient to implement a {\em simplified Newton iteration}, instead of the basic one. This involve using a simplified version of the algorithm shown in Table~\ref{alg1}: it is sketched in Table~\ref{alg2}.

\begin{table}[t]
\caption{Algorithm~\ref{alg2} -- the simplified shooting-Newton procedure.}
\label{alg2}
\hrulefill
\begin{eqnarray*}
{\rm fix~} \rho_0\\
{\rm for ~~} \ell &=& 0,1,\dots\\
                    &&{\rm solve:}~\, y^{(\aa)}(t,\rho_\ell)=f(y(t,\rho_\ell)),\qquad\qquad~\,\, t\in[0,T],\qquad y(0,\rho_\ell) = \rho_\ell\\
                    &&~\,{\rm and ~ compute} \quad     \hat\Phi(T,\rho_\ell) \approx \Phi(T,\rho_\ell)\\[2mm]
                    &&{\rm set:~} \quad \rho_{\ell+1} = \rho_\ell-\hat\Phi(T,\rho_\ell)^{-1}\left[ y(T,\rho_\ell)-\eta\right]\\
                    {\rm end ~~}
\end{eqnarray*}
\hrulefill
\end{table}

For this simplified shooting-Newton procedure, the following result holds true, the proof being similar to that of Theorem~\ref{perron}.

\begin{theo}\label{perron1}
Assume the hypotheses of Theorem~\ref{perron} hold true and that $\hat\Phi(T,\xi)$ is continuously invertible in a neighborhood of $\xi=\rho^*$. Further, assume that the spectral radius of the matrix
$$I-\hat\Phi(T,\rho^*)^{-1}\Phi(T,\rho^*)$$ is less than 1. Then, the algorithm described in Table~\ref{alg2} converges in a suitable neighborhood of the solution $\rho^*$.
\end{theo}

\bigskip
\begin{rem}
However, as one may expect, in this case the results of Theorems~\ref{quadratic} and \ref{linear} does not hold, in general. As matter of fact, only a linear convergence can be granted and, for linear problems, convergence in one iteration cannot be expected, in general.
\end{rem}

\section{Implementing the algorithm}\label{qpol}

Following the approach in \cite{BBBI2024}, let us now explain the way we compute $y(T,\rho_\ell)$ in Algorithms~\ref{alg1} and \ref{alg2}, and $\Phi(T,\rho_\ell)$ in Algorithm~\ref{alg1}.\footnote{On the contrary, $\hat\Phi(T,\rho_\ell)$ in Algorithm~\ref{alg2} is strictly problem dependent and its computation cannot be stated in general: however, a relevant specific case will be considered in Section\,\ref{semilin}.}
From (\ref{sol0}) and (\ref{Phi}), we have to compute:
\begin{equation}\label{yT}
y(T,\rho_\ell) ~=~ \rho_\ell +\frac{1}{\Gamma(\aa)}\int_0^T (T-x)^{\aa-1}f(y(x,\rho_\ell))\dd x,
\end{equation}
and
\begin{equation}\label{PhiT}
\Phi(T,\rho_\ell) ~=~ I + \frac{1}{\Gamma(\aa)}\int_0^T (T-x)^{\aa-1}f'(y(x,\rho_\ell))\Phi(x,\rho_\ell)\dd x.
\end{equation}
To begin with, in order to obtain a piecewise approximation to the solution of the two problems,  we consider a partition of the integration interval in the form:
\begin{equation}\label{tn}
t_n = t_{n-1} + h_n,  \qquad n=1,\dots,N, 
\end{equation}
where
\begin{equation}\label{T}
t_0=0, \qquad t_N= T\equiv \sum_{n=1}^N h_n.
\end{equation} 
In general \cite{BBBI2024}, for coping with possible singularities in the derivative of the vector field at the origin, we shall consider the following {\em graded mesh},
\begin{equation}\label{hn}
h_n = r^{n-1} h_1, \qquad n=1\dots,N,
\end{equation}
where $r>1$ and $h_1>0$ satisfy, by virtue of (\ref{T})-(\ref{hn}),
\begin{equation}\label{rT}
h_1\frac{r^N-1}{r-1} = T.
\end{equation}
Clearly, when a uniform mesh is considered then, in (\ref{hn}), $r=1$ and $h_1=h:=T/N$, so that $h_n=h$, $n=1,\dots,N$.

By setting
\begin{equation}\label{yn}
y_n(ch_n,\rho_\ell) := y(t_{n-1}+ch_n,\rho_\ell), \qquad c\in[0,1],\qquad n=1,\dots,N,
\end{equation}
the restriction of the solution of (\ref{yT}) on the interval $[t_{n-1},t_n]$, and taking into account (\ref{tn})--(\ref{hn}), one then obtains:
\begin{eqnarray}
\nonumber \lefteqn{
y(T,\rho_\ell) ~\equiv~y_N(h_N,\rho_\ell) ~=~ \rho_\ell + \frac{1}{\Gamma(\aa)}\int_0^T (T-x)^{\aa-1} f(y(x,\rho_\ell))\dd x}\\
\nonumber
&=&~\rho_\ell + \frac{1}{\Gamma(\aa)}\sum_{n=1}^N \int_{t_{n-1}}^{t_n} (t_N-x)^{\aa-1}f(y(x,\rho_\ell))\dd x\\[2mm]    
\nonumber
&=&~\rho_\ell + \frac{1}{\Gamma(\aa)}\sum_{n=1}^N \int_0^{h_n} (t_N-t_{n-1}-x)^{\aa-1}f(y_n(x,\rho_\ell))\dd x\\[2mm]  
&=&~\rho_\ell + \frac{1}{\Gamma(\aa)}\sum_{n=1}^N h_n^\aa\int_0^1 
\left(\frac{r^{N-n+1}-1}{r-1}-\tau\right)^{\aa-1}f(y_n(\tau h_n,\rho_\ell))\dd \tau.\qquad \label{yTr}
\end{eqnarray}
In case of a constant stepsize $h=T/N$ is used, the previous equation becomes:
\begin{eqnarray}\nonumber
y(T,\rho_\ell)&\equiv& y_N(h,\rho_\ell)\\
&=& \rho_\ell + \frac{h^\aa}{\Gamma(\aa)}\sum_{n=1}^N \int_0^1 
\left(N-n+1-\tau\right)^{\aa-1}f(y_n(\tau h,\rho_\ell))\dd \tau.\quad \label{yT1}
\end{eqnarray}

Similarly, for (\ref{PhiT}), by setting
\begin{equation}\label{Phin}
\Phi_n(ch_n,\rho_\ell) := \Phi(t_{n-1}+ch_n,\rho_\ell), \qquad c\in[0,1],\qquad n=1,\dots,N,
\end{equation}
the restriction of the solution on the interval $[t_{n-1},t_n]$, again by virtue of (\ref{tn})--(\ref{hn}), one obtains:
\begin{eqnarray}
\nonumber \lefteqn{
\Phi(T,\rho_\ell) ~\equiv~\Phi_N(h_N,\rho_\ell)}\\[2mm] \nonumber
&=& I + \frac{1}{\Gamma(\aa)}\int_0^T (T-x)^{\aa-1} f'(y(x,\rho_\ell))\Phi(x,\rho_\ell)\dd x\\[2mm]
\nonumber
&=&I + \frac{1}{\Gamma(\aa)}\sum_{n=1}^N \int_{t_{n-1}}^{t_n} (t_N-x)^{\aa-1}f'(y(x,\rho_\ell))\Phi(x,\rho_\ell)\dd x\\[2mm]    
\nonumber
&=&I + \frac{1}{\Gamma(\aa)}\sum_{n=1}^N \int_0^{h_n} (t_N-t_{n-1}-x)^{\aa-1}f'(y_n(x,\rho_\ell))\Phi_n(x,\rho_\ell)\dd x\\[2mm]\nonumber  
&=& I + \frac{1}{\Gamma(\aa)}\sum_{n=1}^N h_n^\aa\int_0^1 
\left(\frac{r^{N-n+1}-1}{r-1}-\tau\right)^{\aa-1}f'(y_n(\tau h_n,\rho_\ell))\Phi_n(\tau h_n,\rho_\ell)\dd \tau,\\
 \label{PhiTr}
\end{eqnarray}
and, in case of a constant stepsize $h=T/N$, 
\begin{eqnarray}\nonumber
\lefteqn{\Phi(T,\rho_\ell)~\equiv~ \Phi_N(h,\rho_\ell)}\\
&=& I + \frac{h^\aa}{\Gamma(\aa)}\sum_{n=1}^N \int_0^1 
\left(N-n+1-\tau\right)^{\aa-1}f'(y_n(\tau h,\rho_\ell))\Phi_n(\tau h,\rho_\ell)\dd \tau.\qquad \label{PhiT1}
\end{eqnarray}

\subsection{Piecewise quasi-polynomial approximation}
The previous functions are then approximated via a piecewise quasi-polyno\-mial approximation, as described in \cite{BBBI2024}, which we here briefly recall, and generalize to the approximation of the fundamental matrix solution, too.
In more detail, with reference to (\ref{yn}) and (\ref{Phin}), we shall look for approximations:
\begin{eqnarray}\label{sigPsi}
\sigma_n(ch_n,\rho_\ell) &\simeq& y_n(ch_n,\rho_\ell),\\[2mm] \nonumber
\Psi_n(ch_n,\rho_\ell) &\simeq& \Phi_n(ch_n,\rho_\ell), \qquad c\in[0,1], \qquad n=1,\dots,N,
\end{eqnarray}
and, consequently,
\begin{equation}\label{siThT}
y(T,\rho_\ell) \simeq \sigma_N(h_N,\rho_\ell), \qquad \Phi(T,\rho_\ell) \simeq \Psi_N(h_N,\rho_\ell).
\end{equation}
Following steps similar to those in \cite[Section\,2]{BBBI2024}, we consider the expansion of the vector field along the orthonormal polynomial basis, w.r.t. the weight function
$$\omega(x) = \aa(1-x)^{\aa-1}, \qquad s.t.\qquad \int_0^1 \omega(x)\,\dd x=1,$$
resulting into a scaled and shifted family of Jacobi polynomials:\footnote{Here, $\bar P_j^{(a,b)}(x)$ denotes the $j$th Jacobi polynomial with parameters $a$ and $b$, in $[-1,1]$.}
$$ 
P_j(x) := \sqrt{\frac{2j+\aa}\aa} \bar P_j^{(\aa-1,0)}(2x-1), \qquad x\in[0,1], \qquad j=0,1,\dots.
$$ 
In so doing, for $n=1,\dots,N$, one obtains:
$$f(y_n(ch_n,\rho_\ell)) = \sum_{j\ge0} P_j(c)\gamma_j(y_n,\rho_\ell), \qquad c\in[0,1],$$
with
$$%\begin{equation}\label{gammaj}
\gamma_j(y_n,\rho_\ell) = \aa\int_0^1(1-\tau)^{\aa-1}P_j(\tau)f(y_n(\tau h_n,\rho_\ell))\dd\tau, \qquad j=0,1,\dots.
$$%\end{equation}
The approximations are derived by truncating the infinite series to a finite sum with $s$ terms. Consequently, for $n=1,\dots,N$, one obtains:\footnote{We refer to \cite{ABI2022,BGI2024} for efficient procedures for computing the fractional integrals $I_j^\aa P_j(c)$, $j=0,\dots,s-1.$}
\begin{equation}\label{sign}
\sigma_n(ch_n,\rho_\ell) = \phi_{n-1}^\aa(c,\rho_\ell) + h_n^\aa\sum_{j=0}^{s-1} \gamma_j(\sigma_n,\rho_\ell) I^\aa P_j(c), \qquad c\in[0,1],
\end{equation}
with
\begin{equation}\label{gammaj}
\gamma_j(\sigma_n,\rho_\ell) = \aa\int_0^1(1-\tau)^{\aa-1}P_j(\tau)f(\sigma_n(\tau h_n,\rho_\ell))\dd\tau, \qquad j=0,\dots,s-1.\quad
\end{equation}
and  
\begin{equation}\label{finc}
\phi_{n-1}^\aa(c,\rho_\ell) = \rho_\ell+ \sum_{\nu=1}^{n-1} h_\nu^\aa  \sum_{j=0}^{s-1} J_j^\aa\left(\frac{r^{n-\nu}-1}{r-1}+cr^{n-\nu}\right) \gamma_j(\sigma_\nu,\rho_\ell),
\end{equation}
having set, for $x>1$,\,\footnote{We refer to \cite{BBBI2024,BGI2024} for the efficient computation of such integrals.} 
\begin{equation}\label{Jjaell}
J_j^\aa(x) := \frac{1}{\Gamma(\aa)}\int_0^1(x-\tau)^{\aa-1}P_j(\tau)\dd\tau, \qquad j=0,\dots,s-1.
\end{equation}

If a constant stepsize $h=T/N$ is used, then (\ref{finc}) reads:
\begin{equation}\label{finch}
\phi_{n-1}^\aa(c,\rho_\ell) =\rho_\ell+ h^\aa\sum_{\nu=1}^{n-1} \sum_{j=0}^{s-1} J_j^\aa(n-\nu+c) \gamma_j(\sigma_\nu,\rho_\ell),
\end{equation}
and similarly one modifies (\ref{sign}) and (\ref{gammaj}).

It can be shown (see \cite{BBBI2024}) that $\phi_{n-1}^\aa(c,\rho_\ell)$ is nothing but the approximation of the {\em memory term}
$$ 
G_{n-1}^\aa(c,\rho_\ell)= \rho_\ell + \frac{1}{\Gamma(\aa)}\sum_{\nu=1}^{n-1} h_\nu^\aa\int_0^1 
\left(\frac{r^{n-\nu}-1}{r-1} + cr^{n-\nu}-\tau\right)^{\aa-1}f(y_n(\tau h_n,\rho_\ell))\dd \tau,
$$ 
such that, for all $c\in[0,1]$, and $n=1,\dots,N$:
\begin{equation}\label{ynG}
y_n(ch_n,\rho_\ell) = G_{n-1}^\aa(c,\rho_\ell) + \frac{h_n^\aa}{\Gamma(\aa)}\int_0^c 
\left(c-\tau\right)^{\aa-1}f(y_n(\tau h_n,\rho_\ell))\dd \tau.
\end{equation}
As matter of fact, (\ref{yTr}) corresponds to set $n=N$ and $c=1$ in (\ref{ynG}). 

Similarly, when a constant stepsize $h=T/N$ is used, then
$$
G_{n-1}^\aa(c,\rho_\ell)= \rho_\ell + \frac{h^\aa}{\Gamma(\aa)}\sum_{\nu=1}^{n-1} \int_0^1 
\left(n-\nu+c-\tau\right)^{\aa-1}f(y_n(\tau h,\rho_\ell))\dd \tau,
$$
and  (\ref{ynG}) still formally holds, upon replacing $h_n$ with $h$. Consequently,  (\ref{yT1}) corresponds again to set $n=N$ and $c=1$ in (\ref{ynG}).
\smallskip

The Fourier coefficients (\ref{gammaj}) can be approximated up to machine precision by using the Gauss-Jacobi formula of order $2k$ based at the zeros of $P_k(x)$, $c_1,\dots,c_k$, with corresponding weights $b_1,\dots,b_k$, by choosing  $k$ large enough. As is explained in \cite[Section\,3]{BBBI2024}, this allows formulating the discrete problem for computing them as:\footnote{As is usual, the function $f$, here evaluated in a (block) vector of dimension $k$, denotes the (block) vector made up by $f$ evaluated in each (block) entry of the input argument.}
\begin{equation}\label{vform}
\bfgamma^n = \P_s^\top\Omega \otimes I_m f\left( \bfphi_{n-1}^\aa +h_n^\aa\I_s^\aa\otimes I_m\bfgamma^n\right),
\end{equation}
with, by setting $\gamma_j^n(\rho_\ell)$, $j=0,\dots,s-1$, the approximation to $\gamma_j(\sigma_n,\rho_\ell)$ obtained by using the Gauss-Jacobi quadrature formula,
$$\bfgamma^n = \pmatrix{c} \gamma_0^n(\rho_\ell)\\ \vdots \\ \gamma_{s-1}^n(\rho_\ell) \endpmatrix\in\RR^{sm},\qquad \bfphi_{n-1}^\aa = \pmatrix{c} \phi_{n-1}(c_1,\rho_\ell)\\ \vdots\\ \phi_{n-1}(c_k,\rho_\ell)\endpmatrix\in\RR^{km},$$
and
$$
\P_s = \pmatrix{ccc} P_0(c_1) & \dots & P_{s-1}(c_1)\\ \vdots & &\vdots\\ P_0(c_k) & \dots & P_{s-1}(c_k)\endpmatrix,
\quad \I_s^\aa = \pmatrix{ccc} I^\aa P_0(c_1)& \dots &I^\aa P_{s-1}(c_1)\\ \vdots & &\vdots\\ I^\aa P_0(c_k) & \dots &I^\aa P_{s-1}(c_k)\endpmatrix\quad \in\RR^{k\times s},$$
$$\Omega = \pmatrix{ccc} b_1\\ &\ddots\\ &&b_k\endpmatrix\in\RR^{k\times k}.$$

\begin{rem}\label{noncosta}
It is worth noticing that the discrete problem (\ref{vform}) has (block) dimension $s$, {\em independently of $k$}. This, in turn, allows using relatively large values of $k$, in order to have an accurate approximation of the Fourier coefficients, without increasing too much the computational cost. 

Moreover, the vector $\bfphi_{n-1}^\aa$ in (\ref{vform}) only depends on known quantities, computed at the previous timesteps. 

Further, we observe that also the matrices $\P_s$, $\I_s^\aa$, as well as all the required integrals (\ref{Jjaell}),  can be computed in advance, once for all, and they can be used for each new approximation $\rho_\ell$ in both  Algorithms~\ref{alg1} and \ref{alg2}.  Additionally, it is worth mentioning that, since they only depend on $s,k,\alpha,r$, in principle they could be tabulated, without needing to be evaluated.
\end{rem}

Considering that $$\I^\aa P_j(1) = \frac{1}{\Gamma(\aa)}\int_0^1 (1-x)^{\aa-1}P_j(x)\dd x = \frac{\delta_{j0}}{\Gamma(\aa+1)},\qquad j=0,\dots,s-1,$$
the approximations of the solution at $t_n$ is given by:
\begin{equation}\label{ytn}
y(t_n,\rho_\ell)\simeq \sigma_n(h_n,\rho_\ell)  \equiv \phi_{n-1}^\aa(1,\rho_\ell) + \frac{h_n^\aa}{\Gamma(\aa+1)}\gamma_0^n(\rho_\ell), \qquad n=1,\dots,N.
\end{equation}

According to \cite{BBBI2024} (see also \cite{BGI2024}), we give the following definition.
\begin{defi}\label{fhbvm}  We shall refer to the method defined by (\ref{vform})-(\ref{ytn}), as a {\em Fractional HBVM with parameters $k$ and $s$}, in short {\em FHBVM$(k,s)$}.\end{defi}\smallskip

In particular, from (\ref{siThT}) and (\ref{ytn}) one obtains, considering that $t_N=T$:
\begin{equation}\label{sigmaT}
y(T,\rho_\ell) \simeq \sigma_N(h_N,\rho_\ell) \equiv \phi_{N-1}^\aa(1,\rho_\ell) + \frac{h_N^\aa}{\Gamma(\aa+1)}\gamma_0^N(\rho_\ell).
\end{equation}

\begin{rem}
When $\aa=1$, the polynomials $\{P_j\}_{\ge0}$, become the usual Legendre polynomials orthomormal in $[0,1]$. Consequently, a FHBVM$(k,s)$ method reduces to a standard HBVM$(k,s)$ method, when $\aa=1$.
\end{rem}

In a similar way, for $n=1,\dots,N$:
\begin{equation}\label{Psin}
\Psi_n(ch_n,\rho_\ell) = \Theta_{n-1}^\aa(c,\rho_\ell) + h_n^\aa\sum_{j=0}^{s-1} \Gamma_j(\sigma_n,\rho_\ell) I^\aa P_j(c), \qquad c\in[0,1],
\end{equation}
with
\begin{equation}\label{Gammaj}
\Gamma_j(\sigma_n,\rho_\ell) = \aa\int_0^1(1-\tau)^{\aa-1}P_j(\tau)f'(\sigma_n(\tau h_n,\rho_\ell))\Psi_n(\tau h_n,\rho_\ell)\dd\tau, \qquad  j=0,\dots,s-1,
\end{equation} 
and (see (\ref{Jjaell}))
\begin{equation}\label{Thenc}
\Theta_{n-1}^\aa(c,\rho_\ell) =
I+ \sum_{\nu=1}^{n-1} h_\nu^\aa  \sum_{j=0}^{s-1} J_j^\aa\left(\frac{r^{n-\nu}-1}{r-1}+cr^{n-\nu}\right) \Gamma_j(\sigma_\nu,\rho_\ell).
\end{equation}

Similarly as in (\ref{finch}), when a constant stepsize $h=T/N$ is used, then (\ref{Thenc}) becomes:
\begin{equation}\label{Thench}
\Theta_{n-1}^\aa(c,\rho_\ell) =I+ h^\aa\sum_{\nu=1}^{n-1} \sum_{j=0}^{s-1} J_j^\aa(n-\nu+c) \Gamma_j(\sigma_\nu,\rho_\ell).
\end{equation}

As done for (\ref{gammaj}), by approximating the integrals in (\ref{Gammaj}) by using the same Gauss-Jacobi formula as before, from (\ref{Psin}) and (\ref{Gammaj}) one derives a discrete problem in the form
\begin{equation}\label{vform1}
\bfGamma^n = \P_s^\top\Omega \otimes I_m f'\left( \bfphi_{n-1}^\aa +h_n^\aa\I_s^\aa\otimes I_m\bfgamma^n\right)
\left[ \bfTheta_{n-1}^\aa + h_n^\aa\I_s^\aa\otimes I_m \bfGamma^n\right],
\end{equation}
where $\bfphi_{n-1}^\aa$ and $\bfgamma^n$ have been already computed in (\ref{vform}), 
$$f'\left( \bfphi_{n-1}^\aa +h_n^\aa\I_s^\aa\otimes I_m\bfgamma^n\right)\in\RR^{km\times km}$$ is the block diagonal matrix whose diagonal blocks are given by the corresponding evaluations of the Jacobian of $f$, 
$$\bfTheta_{n-1}^\aa = \pmatrix{c}\Theta_{n-1}^\aa(c_1,\rho_\ell)\\ \vdots \\ \Theta_{n-1}^\aa(c_k,\rho_\ell)\endpmatrix\in\RR^{km\times m},$$ and, by setting $\Gamma_j^n(\rho_\ell)$, $j=0,\dots,s-1$,  the approximation to $\Gamma_j(\sigma_n,\rho_\ell)$ obtained through the Gauss-Jacobi formula,
$$\bfGamma^n =  \pmatrix{c} \Gamma_0^n(\rho_\ell)\\ \vdots \\ \Gamma_{s-1}^n(\rho_\ell) \endpmatrix\in\RR^{sm\times m},$$
with the approximation of the solution at $t_n$ given by:
\begin{equation}\label{Psitn}
\Phi(t_n,\rho_\ell) \simeq \Psi_n(h_n,\rho_\ell) \equiv \Theta_{n-1}^\aa(1,\rho_\ell) + \frac{h_n^\aa}{\Gamma(\aa+1)}\Gamma_0^n(\rho_\ell),\qquad n=1,\dots,N.
\end{equation}
As is clear, (\ref{vform1})-(\ref{Psitn}) define the application of the FHBVM$(k,s)$ method to the variational problem. 

We observe that considerations similar to those made in Remark~\ref{noncosta} for (\ref{vform}) can be now repeated for (\ref{vform1}),  with the additional fact that (\ref{vform1}) amounts to just solving a {\em linear system of equations}. 

At last, from (\ref{siThT}) and (\ref{Psitn}) one eventually obtains:
\begin{equation}\label{PsiT}
\Phi(T,\rho_\ell) \simeq \Psi_N(h_N,\rho_\ell) \equiv \Theta_{N-1}^\aa(1,\rho_\ell) + \frac{h_N^\aa}{\Gamma(\aa+1)}\Gamma_0^N(\rho_\ell).
\end{equation}

\smallskip 
\begin{rem}\label{spectral}
By choosing values of $s$, and $k\ge s$, large enough, it can be seen that the approximations (\ref{sigmaT}) and (\ref{PsiT}) provided by a FHBVM$(k,s)$ method can be accurate up to machine precision.  In fact, from the analysis carried out in  \cite{BBBI2024}, the error in approximating (\ref{yT}) and (\ref{PhiT}) is proved to be bounded  by
$$O(h_1^{2\aa}+h_N^{s+\aa}),$$ if a graded mesh (\ref{tn})--(\ref{hn}) is used, or by $$O(h^{s+\aa-1}),\qquad h=T/N,$$ if a uniform mesh can be considered. This latter case is appropriate when the vector field is everywhere smooth, in a neighborhood of the solution.

Actually, this amounts to using the method as a {\em spectrally accurate method in time}, as is the case for HBVMs  \cite{ABI2020,ABI2023,BIMR2019,BMR2019}. This kind of  approximations will be considered in the implementation of the algorithm listed in Table~\ref{alg1} (and for the simplified version of it, listed in Table~\ref{alg2}), which we shall use for the numerical tests reported in Section\,\ref{num}.
\end{rem}

\subsection{Error estimation}\label{errest}
It is worth mentioning that the procedure explained in the previous section allows to derive, as a by-product, an estimate for the error in the computed solution, due to the fact that, in Algorithm~\ref{alg1}, the iteration is stopped when, for a suitably small tolerance $tol$,
\begin{equation}\label{stop}
|\rho_{\ell+1}-\rho_\ell|\le tol.
\end{equation}
In fact, in such a case, one expects that $|\rho_{\ell+1}-\rho^*|\approx tol$ as well. Consequently, by considering that at the mesh points, for $n=1,\dots,N$:
\begin{equation}\label{sol}
y(t_n,\rho_\ell) \simeq  \sigma_n(h_n,\rho_\ell) \equiv \phi_{n-1}^\aa(1,\rho_\ell) + \frac{h_n^\aa}{\Gamma(\aa+1)}\gamma_0^n(\rho_\ell),
\end{equation}
and, similarly,
$$
\Phi(t_n,\rho_\ell) \simeq \Psi_n(h_n,\rho_\ell) \equiv \Theta_{n-1}^\aa(1,\rho_\ell) + \frac{h_n^\aa}{\Gamma(\aa+1)}\Gamma_0^n(\rho_\ell),
$$
by virtue of the perturbation result of Theorem\,\ref{fundmat}, one derives the estimates
\begin{equation}\label{erro}
\|y(t_n,\rho^*)-y(t_n,\rho_\ell)\|\approx 2\cdot tol \cdot \|\Psi_n(h_n,\rho_\ell)\|, \qquad n=1,\dots,N.
\end{equation}

\subsection{The simplified shooting-Newton algorithm for semi-linear problems}\label{semilin}
For the simplified algorithm in Table~\ref{alg2}, the approximation to $\hat\Phi(T,\rho_\ell)$ is in general problem dependent. However, there is a specific case where an efficient approximation can be readily obtained, i.e., when problem (\ref{tvp}) is semi-linear:
\begin{equation}\label{semi}
y^{(\aa)}(t) = Ly(t) + g(y(t)), \qquad t\in[0,T], \qquad y(T)=\eta\in\RR^m,
\end{equation} 
with $L\in\RR^{m\times m}$ and $\|L\|\gg\|g\|$ in a suitable neighborhood of the solution. In fact, in such a case, one can approximate the variational problem (\ref{varpro}) with the linear part only (thus, independent of $\rho$),
$$%\begin{equation}\label{varpro1}
\hat\Phi^{(\aa)}(t) = L\,\Phi(t), \qquad t\in[0,T], \qquad \Phi(0) = I.
$$%\end{equation}
In so doing, one obtains the approximation
$$\hat\Phi(T) = E_\aa( LT^\aa) := \sum_{j\ge0} \frac{(L T^\aa)^j}{\Gamma(\aa j+1)},$$
with $E_\aa$ the one-parameter Mittag-Leffler function. Furthermore, since we are interested in deriving only a convenient approximation to the fundamental matrix function, we can truncate the above series to a suitable finite sum. As an example,  for a given tolerance $\varepsilon$, one may consider the approximation:
\begin{equation}\label{hPhi}
\hat\Phi(T) = \sum_{j=0}^J\frac{(L T^\aa)^j}{\Gamma(\aa j+1)},\qquad
s.t. \qquad \frac{\|(L T^\aa)^ J\|}{\Gamma(\aa J+1)} \le \varepsilon.
\end{equation}

\section{Numerical Tests}\label{num}%\marginpar{\blue fde\_var.m \\ Jjalfa1.m \\ notree- \\ jacobi.m}
We here report a few numerical tests aimed at illustrating the theoretical findings. For all tests, we use $k=22$ and $s=20$, so that we are going to use a FHBVM(22,20) method. In other words, we use  a local polynomial approximation of degree $s-1=19$ for the vector field, coupled with a Gauss-Jacobi quadrature formula of order $2k=44$ for approximating the Fourier coefficients (\ref{gammaj}) and (\ref{Gammaj}). We have used straightforward fixed-point iterations, derived from (\ref{vform}) and (\ref{vform1}), respectively, to solve the corresponding discrete problems.\footnote{We have used a fixed point iteration also for solving (\ref{vform1}), despite the fact that it is just a linear system of equations.}  Namely,
$$
\bfgamma^{n,j+1} = \P_s^\top\Omega \otimes I_m f\left( \bfphi_{n-1}^\aa +h_n\I_s^\aa\otimes I_m\bfgamma^{n,j}\right), \qquad j=0,1,\dots,
$$
starting from $\bfgamma^{n,0}=\bfzero$, for (\ref{vform}), and similarly for (\ref{vform1}).\footnote{More refined nonlinear iterations are described in \cite{BGI2024}.} 
The iterations are carried out until full machine accuracy is gained, so that we expect full machine accuracy for the computed approximation (\ref{PsiT}) to $\Phi(T,\rho_\ell)$, as well as a corresponding fully accurate discrete solution (\ref{sol}).  

We consider 6 test problems:
\begin{itemize}
\item the first 3 problems are the same scalar test problems in \cite[Section\,5]{DiUl2023};
\item the last 3 problems are vector problems.
\end{itemize}
For all problems, (see (\ref{tvp})) the initial guess $\rho_0=\eta$ has been considered.
All numerical tests have been performed in Matlab$^\copyright$ (Rel.\,2023b) on a Silicon M2 laptop with 16GB of shared memory. The iteration of Algorithm~\ref{alg1}  is stopped by using a tolerance $tol=10^{-14}$ in (\ref{stop}).
The same tolerance and stopping criterion will be used for Algorithm~\ref{alg2}. To be more precise, we shall consider Algorithm~\ref{alg1} for solving all the problems, and Algorithm~\ref{alg2} for solving the last problem, which is semi-linear.

\subsection{Example~1} %\marginpar{\blue prob12.m}
The first problem is given by:
\begin{eqnarray}\nonumber
y^{(0.3)} &=& -|y|^{1.5} +\frac{8!}{\Gamma(8.7) }t^{7.7} -
           3\frac{\Gamma(5.15)}{\Gamma(4.85)}t^{3.85} +
           \left( \frac{3}2t^{0.15} - t^4 \right)^3 + \frac{9}4\Gamma(1.3),\\[2mm] \label{prob1} &&\qquad t\in[0,1], \qquad y(1)=\frac{1}4,
\end{eqnarray}
whose solution is
$$%\begin{equation}\label{prob1sol}
y(t) = t^8-3\,t^{4.15}+\frac{9}4\,t^{0.3}.
$$%\end{equation}
In this case, we use a uniform mesh with stepsize  $h=1/10$. The method converges in 4 iterations producing the approximations in Table~\ref{tabex1}.
\begin{table}
\caption{Results for Problem (\ref{prob1}).}
\label{tabex1}
\centering
\begin{tabular}{|r|r|}
\hline
$\ell$ & $\rho_\ell$\hspace{1.5cm} \\
\hline
0&     2.500000000000000e-01\\
1&    -6.974105632991501e-03\\
2&    -6.267686473630449e-06\\
3&    -5.040632537594832e-12\\
4&    -2.508583045846617e-15\\
%5&    -1.623592639776642e-15\\
\hline
\end{tabular}
\end{table}
It is possible to appreciate the quadratic convergence of the iteration in the first iterations (in the last one, roundoff errors clearly dominate). 
The maximum error on the final solution is $\approx 6\cdot 10^{-15}$, whereas the estimated one, by using (\ref{erro}), is $2\cdot 10^{-14}$.

\subsection{Example~2} %\marginpar{\blue prob13.m, ml.m}
The second problem is given by:
\begin{eqnarray}\nonumber
y^{(0.3)} &=& -\frac{3}2 \,y, \qquad t\in[0,7], \\[2mm] \label{prob2}
 y(7) &=& \frac{14}5 E_{0.3}\left(-\frac{3}2 \,7^{\,0.3}\right)\simeq .6476128469955936,
\end{eqnarray}
with $E_{0.3}$ the Mittag-Leffler function of order 0.3, with solution
$$%\begin{equation}\label{prob2sol}
y(t) =  \frac{14}5 E_{0.3}\left(-\frac{3}2 \, t^{\,0.3}\right).
$$%\end{equation}
We refer to \cite{Ga2015ml} and the accompanying software\, {\tt ml.m}, for an efficient Matlab$^\copyright$ implementation of the Mittag-Leffler function.

In this case a uniform mesh is not appropriate, since the vector field is proportional to the solution, which has a singularity in the first derivative at the origin. Consequently, we use a graded mesh, according to (\ref{hn}), with $h_1=10^{-14}$ and $N=500$.
Taking into account (\ref{rT}), this implies $r\simeq 1.064914852480467$.  According to the result of Theorem~\ref{linear}, the method converges in one iteration, as is shown in Table~\ref{tabex2}.
\begin{table}
\caption{Results for Problem (\ref{prob2}).}
\label{tabex2}
\centering
\begin{tabular}{|r|r|}
\hline
$\ell$ & $\rho_\ell$\hspace{1.5cm} \\
\hline
0  &   .6476128469955936\\
1  &   2.799999999999968\\
%2  &   2.799999999999979\\
\hline
\end{tabular}
\end{table}
The maximum error on the final solution is $\approx 2\cdot 10^{-13}$, whereas the estimated one, by using is (\ref{erro}), is $2\cdot 10^{-14}$ (in this case, the maximum error is essentially close to the origin, where there is the singularity of the derivative).

\subsection{Example~3} %\marginpar{\blue prob11.m}
The third problem is given by:
\begin{eqnarray}\nonumber
y^{(0.7)} &=& \frac{1}{t+1}\sin(t\cdot y), \qquad t\in[0,20], \\[2mm] \label{prob3}
 y(20) &=& 0.8360565285776644,
\end{eqnarray}
which corresponds to the initial value $y(0)=1$. In such a case, the solution is not known in closed form, and the final value has been taken from a reference solution computed by using the FHBVM(22,20) method with a constant stepsize $h=0.02$ (i.e., by using 1000 timesteps). This solution is depicted in Figure~\ref{prob3fig}, and the estimated error (by using a doubled mesh) is $\approx 1.8\cdot 10^{-14}$. %\marginpar{\blue prob3fig.eps}

\begin{figure} 
\centering
\includegraphics[width=10cm]{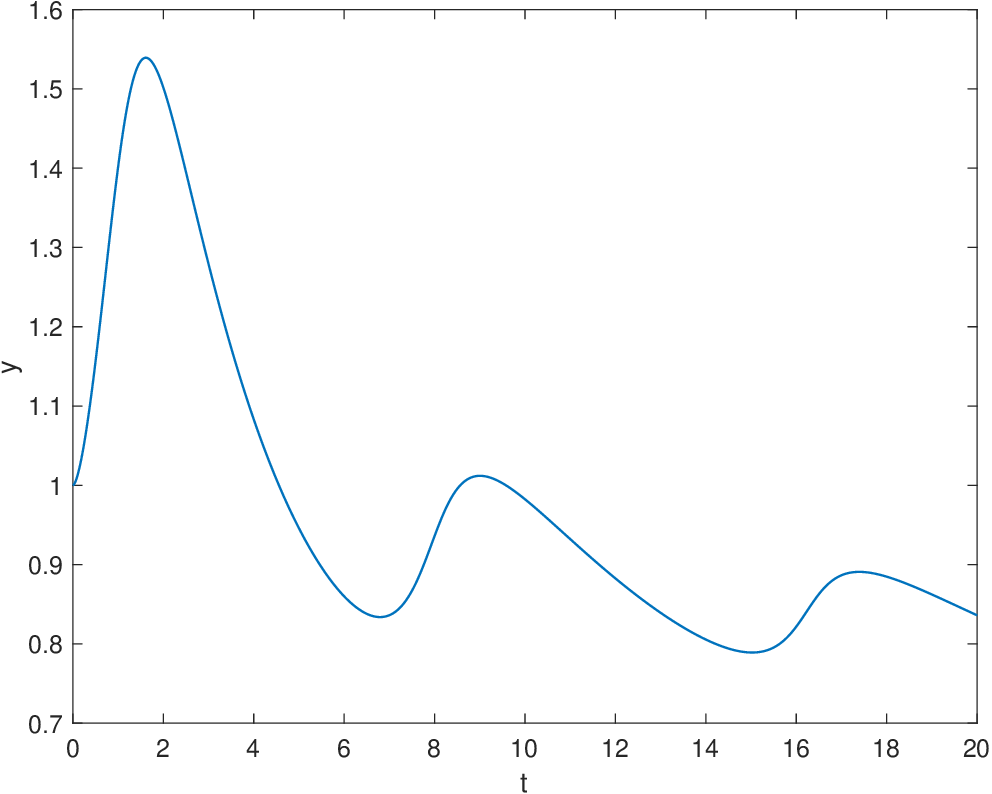}
\caption{Reference solution for problem (\ref{prob3}).}
\label{prob3fig}
\end{figure}

For solving problem (\ref{prob3}), we use a uniform mesh with stepsize $h=20/400=1/20$. The method converges in 6 iterations, producing the approximations listed in Table~\ref{tabex3}.
\begin{table}
\caption{Results for Problem (\ref{prob3}).}
\label{tabex3}
\centering
\begin{tabular}{|r|r|}
\hline
$\ell$ & $\rho_\ell$\hspace{1.5cm} \\
\hline
0&     .8360565285776644\\
1&     1.115178544783084\\
2&     1.057854760373079\\
3&     1.006528883050734\\
4&     .9999714859685488\\
5&     .9999999991678453\\
6&     .9999999999999855\\
%     9.999999999999859e-01
\hline
\end{tabular}
\end{table}
Also in the case, it is possible to appreciate a quadratic-like convergence of the iteration. 
The maximum error in the final solution is $\approx 2\cdot 10^{-14}$, whereas the estimated one, by using is (\ref{erro}), is $6\cdot 10^{-14}$.

\subsection{Example~4} % sistema.m
We now consider the following linear (vector) FDE-TVP:
\begin{eqnarray}\label{prob4}
y^{(0.5)} &=& \pmatrix{rr}-3 & ~0 \\ -2 &~-1\endpmatrix y, \qquad t\in[0,2],\\[2mm]
y(2)        &=& \pmatrix{c} 2\,E_{0.5}\left(-3\cdot 2^{0.5}\right)\\[1mm]
2\,E_{0.5}\left(-3\cdot 2^{0.5}\right)+E_{0.5}\left(-2^{0.5}\right)\endpmatrix 
\simeq \pmatrix{r} .2591172572977875   \\[1mm] .5953212597441289 \endpmatrix,\nonumber
\end{eqnarray}
having solution
$$y(t) = \pmatrix{c} 2\,E_{0.5}\left(-3\cdot t^{0.5}\right)\\[1mm]
2\,E_{0.5}\left(-3\cdot t^{0.5}\right)+E_{0.5}\left(- t^{0.5}\right)\endpmatrix,$$
corresponding to the initial value $y(0) = (2,\,3)^\top$. Since the vector field is linearly related to the solution, which has a 
singularity in the first derivative at the origin, we use a graded mesh with $h_1=10^{-14}$ and $N=100$.  According to the result of Theorem~\ref{linear}, convergence is gained in just one iteration, as is confirmed by Table~\ref{tabex4}.
\begin{table}
\caption{Results for Problem (\ref{prob4}).}
\label{tabex4}
\centering\begin{tabular}{|r|rr|}
\hline
$\ell$ & \multicolumn{2}{c|}{$\rho_\ell$} \\
\hline
 0 &    .2591172572977875   &  .5953212597441289\\
 1 &    2.000000000000012    &  3.000000000000012\\
 %2 &    2.000000000000016    & 3.000000000000015\\
\hline
\end{tabular}
\end{table}
The maximum error in the final solution is $\approx 7\cdot 10^{-15}$, whereas the estimated one, by using  (\ref{erro}), is $2\cdot 10^{-14}$.

\subsection{Example~5} % sistema2.m

\begin{figure}[t]
\centering
\includegraphics[width=10cm]{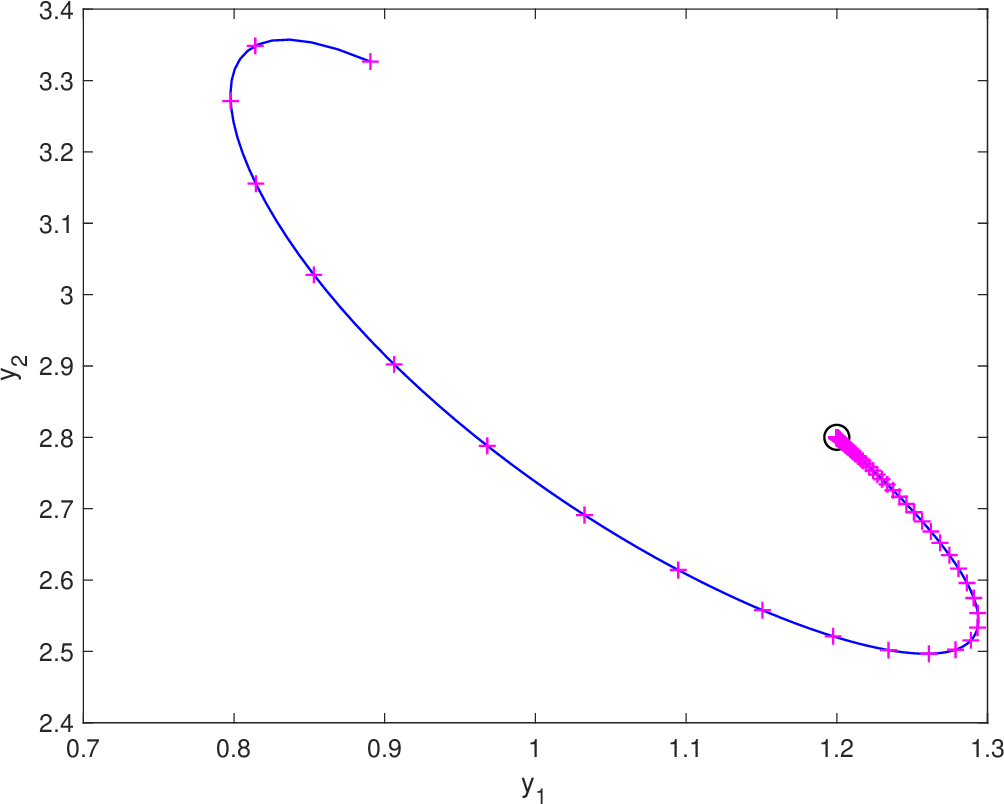}
\caption{Reference solution for problem (\ref{prob5}) (solid line). The circle denotes the actual initial condition, whereas the pluses denote the final approximate solution.}
\label{prob5fig}
\end{figure}

We now consider the following fractional Brusselator model:
\begin{eqnarray}\label{prob5}
y_1^{(0.7)} &=& 1-4y_1+y_1^2y_2,\\[2mm] \nonumber
y_2^{(0.7)} &=& 3y_1-y_1^2y_2, \qquad t\in[0,5],\\[2mm]
y(5) &=& \pmatrix{c}.8904632063462272\\[1mm]  3.326603532694057 \endpmatrix.\nonumber
\end{eqnarray}
In such a case, the solution is not explicitly known, and we have computed the final value starting from
$y(0)=(1.2,\,2.8)^\top$ by using the FHBVM(22,20) method with a graded mesh with $h_1=10^{-14}$ and $N=1000$: the reference solution is plotted in Figure~\ref{prob5fig} in solid line, with the initial condition marked by the circle. We solve the problem by using the FHBVM(22,20) method on a graded mesh with $h_1=10^{-14}$ and $N=200$. In so doing, the algorithm described in Table~\ref{alg1} converges in  5 iterations, with a quadratic-like order, obtaining the results
listed in Table~\ref{tabex5}.
\begin{table}
\caption{Results for Problem (\ref{prob5}).}
\label{tabex5}
\centering\begin{tabular}{|r|rr|}
\hline
$\ell$ & \multicolumn{2}{c|}{$\rho_\ell$} \\
\hline
0&     .8904632063462272&     3.326603532694057\\
1&     1.195221947994766&     2.798766749634182\\
2&     1.199608077826518&     2.800213499824565\\
3&     1.199998157974212&     2.800001859877902\\
4&     1.199999999973615&     2.800000000034993\\
5&     1.199999999999924&     2.800000000000298\\
%6&     1.199999999999925&     2.800000000000296\\
\hline
\end{tabular}
\end{table}
The maximum estimated error in the final solution is $\approx 10^{-13}$, whereas that in the final point is
$\approx 4\cdot 10^{-16}$.

\subsection{Example~6} % sistemaN31.m

 As a last example, we consider a family of semi-linear problems with $y\in\RR^{2\nu}$ and
\begin{equation}\label{ex6}  % sistemaN31.m, testaN1.m
y^{(0.7)} = \pmatrix{cc} & I_\nu \\-I_\nu\endpmatrix y +  \frac{1}{20}\cos(D_\nu y), \qquad t\in[0,5],
\end{equation}
where $I_\nu\in\RR^{\nu\times \nu}$ is the identity matrix, the function $\cos$ is meant to be applied in vector mode, and
$$D_\nu = \diag\left( 1,\, 2,\,\dots, 2\nu\right)^{-1}.$$ 
The reference solution at $t=5$ has been computed by using the FHBVM(22,20) method on a graded mesh with $N=300$ and $h_1=10^{-14}$, solving (\ref{ex6}) starting from the initial value with entries:
\begin{equation}\label{ex6y0}
y_i(0) = \frac{1}i\cos\left( (i-1)\frac{\pi}\nu\right), \qquad i=1,\dots,2\nu. 
\end{equation}
We solve, at first, the FDE-TVP (\ref{ex6}) with $y(5)$ given, by using Algorithm~\ref{alg1} with the FHBVM(22,20) method on a graded mesh with $N=35$ and $h_1=10^{-8}$, for $\nu=1,\dots,35$, thus solving FDE-TVPs having dimension 2,\,4,\,\dots,\,70. 

The algorithm in Table~\ref{alg1} turns out to always converge in 4--5 iterations. The error in the computed initial value is always less than  $1.5\cdot 10^{-13}$. In Figure~\ref{fig3} is the plot of the execution mean times (over 5 runs) of the algorithm versus the dimension of the problem. In more detail, the figure plots:
\begin{itemize}
\item the total execution time (all times are in {\tt sec});
\item the time for computing the required memory terms  $\phi_{n-1}^\aa(c,\rho_\ell)$ (\ref{finc})  in the local problems (LPs);
\item the time for solving the local problems (\ref{vform}) ;
\item the time for computing the memory terms  (\ref{Thenc})  in the local variational problems (LVPs);
\item the time for solving the local variational problems (\ref{vform1}).
\end{itemize}  
According to Remark~\ref{noncosta}, we have not considered the pre-processing time for evaluating the integrals $I^\aa P_j$ in (\ref{sign}) and $J_j^\aa(x)$ (\ref{Jjaell}), also because they require an extended precision arithmetic (quadruple precision would be enough) but, at the moment, they are computed symbolically in Matlab, and not numerically, so that this part of the code is not yet optimized.

From the obtained results, one may conclude that most of the computational time  of Algorithm~\ref{alg1} is spent in the solution of the variational problem: in particular, the evaluation of the memory terms for the local variational problems. 
 For this reason, we now consider Algorithm~\ref{alg2} for solving problem (\ref{ex6}). In fact, since the problem is in the form (\ref{semi}), we can use the (quite cheap) approximation (\ref{hPhi}) in place of the fundamental matrix function. Having fixed a tolerance $\varepsilon=10^{-10}$, this results in using $J=40$ in (\ref{hPhi}), which is quite inexpensive. In such a case, the algorithm in Table~\ref{alg2} converges in 9--10 iterations, instead of 4--5.  
Nevertheless, the overall execution time results to be relatively small, due to the fact that the solution of the variational problem is no more required.  Figure~\ref{fig4} contains the comparison between the total execution times of Algorithm~\ref{alg1}, as seen in Figure~\ref{fig3}, and of Algorithm~\ref{alg2}: this latter is used for solving problem (\ref{ex6})  with $\nu=5,10,15,20,\dots,405$. The highest dimension ($2\nu=810$) is chosen because the corresponding execution time is practically the same as that of Algorithm~\ref{alg1} when solving the problem of dimension 70 (about 7.5 sec). This  clearly shows the superiority of the simplified  shooting-Newton iteration over the original one, for this semi-linear problem.

\begin{figure}[ph]
\centering
\includegraphics[width=10cm]{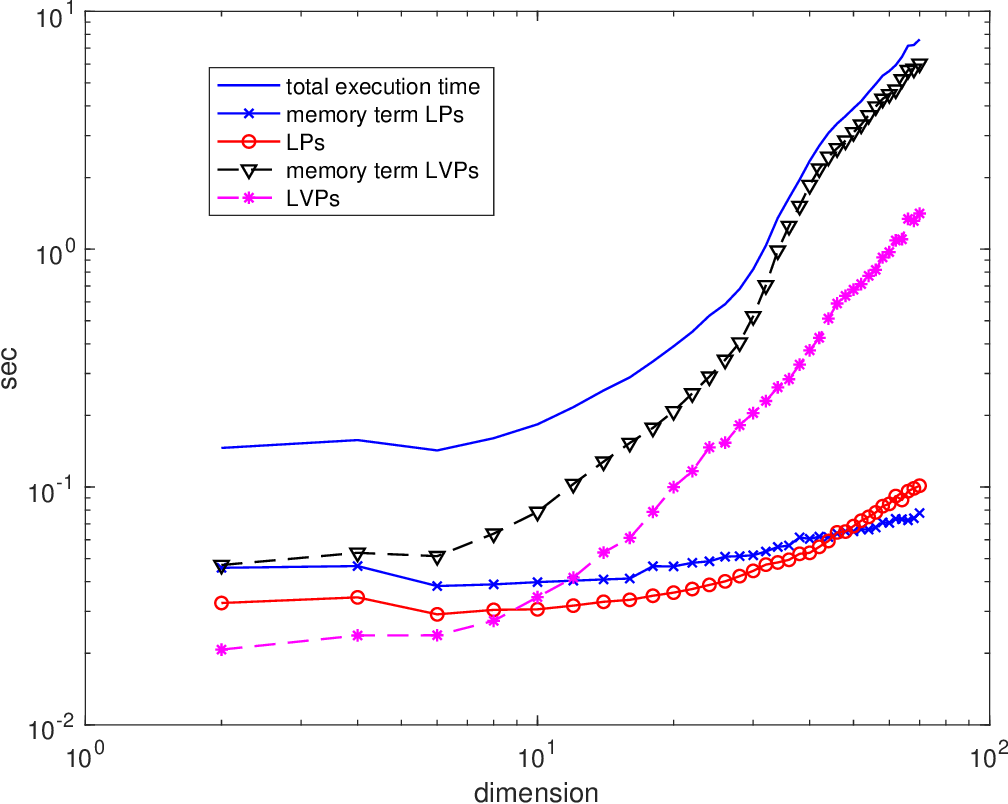}
\caption{Execution times of Algorithm~\ref{alg1} for solving problem (\ref{ex6}) with $y(5)$ given, for dimensions ranging from 2 to 70. See the text for details.}
\label{fig3}
%\end{figure}
\bigskip\bigskip

%\begin{figure}[t]
%\centering
\includegraphics[width=10cm]{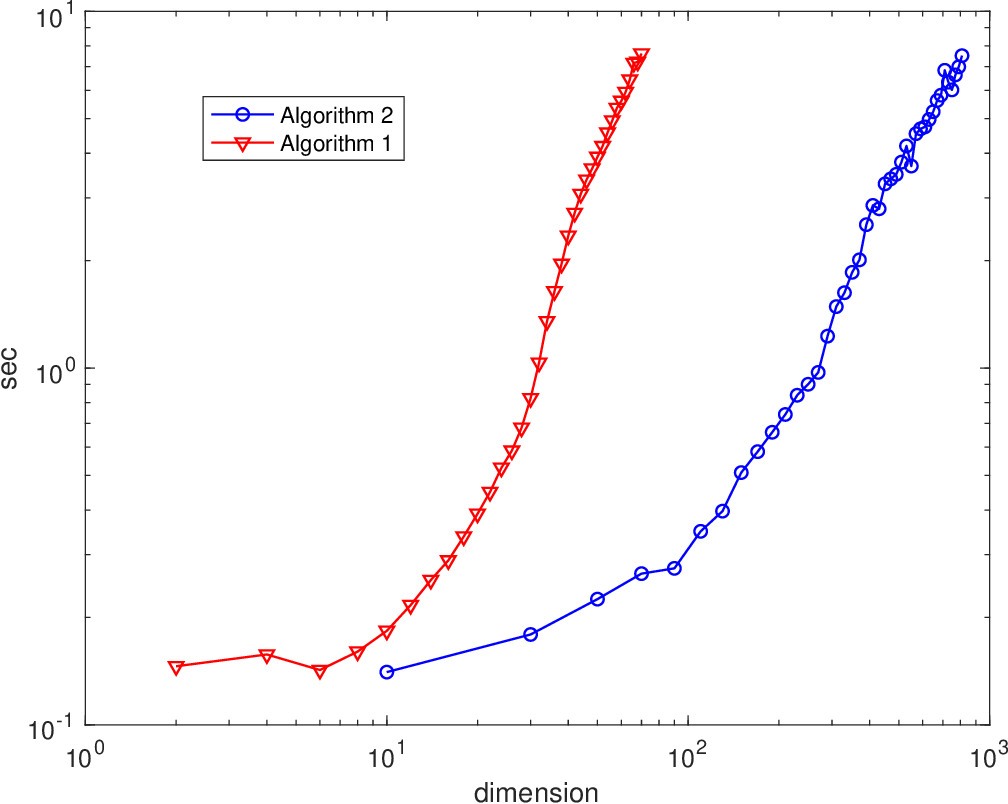}
\caption{Comparison of the total execution times of Algorithm~\ref{alg1} and Algorithm~\ref{alg2} for solving problem (\ref{ex6}) with $y(5)$ given. See the text for details.}
\label{fig4}
\end{figure}

\section{Conclusions}\label{fine}
In this paper we have described a novel shooting procedure which, coupled with the Newton method, proves very appealing for numerically solving  terminal value problems for fractional differential equations. The implementation details of the given procedure have been thoroughly given, when the underlying numerical methods are FHBVMs. These latter methods, when used as  spectrally accurate methods in time, allow deriving very accurate solutions, along with a suitable estimate of the error in the computed solution.  

A corresponding cheaper procedure, relying on a simplified Newton method, has been also described. This latter procedure appears to be very promising for semi-linear problems since, in such a case, the associated variational equation is no more required. 
Numerical tests on both scalar and vector problems confirm the effectiveness of the presented approach. 

Further directions of investigations include the extension for solving two-point boundary value problems, as well as the efficient numerical solution of the local variational problems, due to the fact that they amount to solving just linear systems of algebraic equations.

%\paragraph*{Declarations.} The authors declare no conflict of interests, nor competing interests. No funding was received for conducting this study.

%\paragraph*{Data availability.} All data reported in the manuscript have been obtained by a Matlab$^\copyright$ implementation of the methods presented. They can be made available on request.

%\paragraph*{Acknowledgements.} The authors want to thank the referees for carefully reading the manuscript.

\end{document}